\documentclass[a4paper,11pt]{article}

\usepackage[margin=2.5cm]{geometry}
\usepackage{multirow}
\usepackage{graphicx}
\usepackage{amsmath}
\usepackage{amsthm,amsfonts,amssymb,times} 
\usepackage{enumitem}
\usepackage{tikz}
\usepackage{authblk}
 
\usepackage{thmtools,thm-restate}
\usepackage{cite}
\usepackage{booktabs}
\usetikzlibrary{patterns}
\usetikzlibrary{decorations.markings}
\usetikzlibrary{arrows}
\usepackage{hyperref}
\usepackage{xcolor}

\hypersetup{
    pdftitle=   {Fragile minor-monotone parameters under a random edge perturbation},
   pdfauthor=  {Dong Yeap Kang, Mihyun Kang, Jaehoon Kim, Sang-il Oum},
   colorlinks,
   linkcolor={red!60!black},
   citecolor={green!60!black},
   urlcolor={blue!60!black}}

\usepackage[mathlines]{lineno} 

\newtheorem{THM}{Theorem}[section]
\newtheorem*{THM*}{Theorem~\ref{main}}
\newtheorem{LEM}[THM]{Lemma}
\newtheorem{OBS}[THM]{Observation}

\newtheorem{CLAIM}{Claim}

\newtheorem{PROB}{Problem}

\newcommand\tw{\operatorname{tw}}

\newcommand\abs[1]{\lvert #1\rvert}

\newenvironment{subproof}[1][\proofname]{%
  \begin{proof}[#1]%
}{%
  \end{proof}%
}

\newcommand{\COMMENT}[1]{}

\tikzset{middledownarrow/.style={
        decoration={markings,
            mark= at position 0.4 with {\arrow{#1}} ,
        },
        postaction={decorate}
    }
}

\tikzset{middlearrow/.style={
        decoration={markings,
            mark= at position 0.6 with {\arrow{#1}} ,
        },
        postaction={decorate}
    }
}

\tikzset{middleuparrow/.style={
        decoration={markings,
            mark= at position 0.75 with {\arrow{#1}} ,
        },
        postaction={decorate}
    }
}

\tikzset{middleupuparrow/.style={
        decoration={markings,
            mark= at position 0.9 with {\arrow{#1}} ,
        },
        postaction={decorate}
    }
}

\tikzset{middleupupuparrow/.style={
        decoration={markings,
            mark= at position 0.95 with {\arrow{#1}} ,
        },
        postaction={decorate}
    }
}

\tikzset{middleupupupuparrow/.style={
        decoration={markings,
            mark= at position 1 with {\arrow{#1}} ,
        },
        postaction={decorate}
    }
}
\theoremstyle{definition}
\newtheorem{DEFN}[THM]{Definition}
\newtheorem{EX}[THM]{Example}

\numberwithin{equation}{section}

\begin{document}

\title{Fragile minor-monotone parameters under~a~random~edge~perturbation}
\author[1]{Dong Yeap Kang%
\thanks{This project has received funding from the European Research
Council (ERC) under the European Union's Horizon 2020 research and innovation programme (grant agreement no. 786198, Dong Yeap Kang). He was also supported by the Institute for Basic Science (IBS-R029-Y6).}}
\author[2]{Mihyun Kang%
\thanks{Supported in part by the Austrian Science Fund (FWF) [10.55776/F1002]. 
}
}
\author[4]{Jaehoon Kim%
\thanks{Supported by the POSCO Science Fellowship of POSCO TJ Park Foundation,  by the KAIX Challenge program of KAIST Advanced Institute for Science-X, and by the National Research Foundation of Korea (NRF) grant funded by the Korea government(MSIT) No.\ RS-2023-00210430.}}
\author[3,4]{Sang-il Oum%
\thanks{Supported by the Institute for Basic Science (IBS-R029-C1).}}
 
\affil[1]{\small Extremal Combinatorics and Probability Group, Institute for Basic Science (IBS), Daejeon, South Korea}
\affil[2]{\small Institute of Discrete Mathematics, Graz University of Technology, Graz, Austria}
\affil[3]{\small Discrete Mathematics Group, Institute for Basic Science (IBS), Daejeon, South Korea}
\affil[4]{\small Department of Mathematical Sciences, KAIST, Daejeon, South Korea}

\affil[ ]{\small Email: \texttt{dykang.math@ibs.re.kr}, 
\texttt{kang@math.tugraz.at},
\texttt{jaehoon.kim@kaist.ac.kr},
\texttt{sangil@ibs.re.kr}}

\date{\today}
\maketitle

\begin{abstract}
We conduct a quantitative analysis of how many random edges need to be added to a base graph~$H$ in order to significantly increase natural minor-monotone graph parameters of the resulting graph~$R$.
Specifically, we show that
if $R$ is obtained from a connected graph $H$ by adding only a few random edges, 
the tree-width, genus, and Hadwiger number of $R$ become very large, 
irrespective of the structure of~$H$.


\end{abstract}

\section{Introduction}\label{sec:intro}
Since the seminal work of Erd\H{o}s and R\'enyi~\cite{erdos1960}, various random graph models have been introduced and investigated. The most well-known model is the binomial random graph $G(n,p)$, a graph on the vertex set $[n]:=\{1,2,\ldots,n\}$ where each unordered pair of distinct vertices is connected by an edge with probability $p\in [0,1]$ independently. In recent years, randomly perturbed graphs have received considerable attention (see \cite{ADHL2022, ADHL2022_2, AHK2023, AHK2025, balogh2019, boettcher2018, boettcher2019, BPSS2023, bohman2003, das2019, DMT2020, hahn_klimroth2020, HMT2021, joos2019, KKL20,  krivelevich2017, parczyk2019, powierski2019} and references therein). In this paper, we study a randomly perturbed graph defined as follows: Given an $n$-vertex graph $H:=H_n$, called the \emph{base graph}, we form a new graph $R:=H \cup G(n,p)$ by adding to $H$ the edges of $G(n,p)$ defined on the same vertex set as $H$.
We extend and strengthen a result of Dowden, Kang, and Krivelevich~\cite{dowden2019} about the `fragile genus'  property of~$R$ to additional minor-monotone graph parameters.

Our first result says that the tree-width $\tw(\cdot)$, 
the genus $g(\cdot)$, and the Hadwiger number $h(\cdot)$ are `fragile' in the sense that adding only a few random edges to a base graph can cause a drastic increase in these parameters. To state this formally, we first introduce some necessary notation. Throughout the paper, we use the standard Landau notation for asymptotic orders, with all asymptotics considered as $n\rightarrow \infty$.  For example, whenever we write $x=o(1)$ and $y=\omega(1)$, we mean that $x$ tends to $0$ and $y$ tends to $\infty$ as $n \rightarrow \infty$. We say that an event $\mathcal E$ holds with high probability (\emph{whp} for short) if $\mathcal E$ holds with probability tending to $1$ as $n \rightarrow \infty$. We use $\log$ to denote the natural logarithm unless the base is explicitly stated.  
\begin{restatable}{THM}{fragilemain}
\label{thm:fragile-main}
Assume that $p = p(n) \in (0,1]$ and $\Delta := \Delta(n)\in [1,\infty)$ 
such that  
$n^2 p = \omega(1)$, and $\Delta \leq n^2 p/48000$.
Let $H:=H_n$ be an $n$-vertex connected graph with maximum degree at most $\Delta$, and let $R:=H \cup G(n,p)$. Then whp the following hold:
\begin{enumerate}[label=\rm (\alph*)]
\item$\tw(R)    =   \Omega \left( \tw(H) + \min \left (\frac{n^2 p}{\Delta},n \right ) \right)$;
\item$g(R)     =    \Omega \left(g(H) +  \min\left( \left (\frac{n^2 p}{\Delta} \right )^2,n^2 p\right)\right)$;
\item$h(R) =   \Omega\left( h(H)+\min\Bigl( \sqrt{\frac{n^2p}{\log{\Delta}}}, \frac{n^2p}{\Delta \sqrt{\log \Delta}}  \Bigr) \right)$.
\end{enumerate}
\end{restatable}
We remark that Dowden, Kang, and Krivelevich~\cite{dowden2019} proved a weaker version of Theorem~\ref{thm:fragile-main} (b) for a base graph $H$ with \emph{bounded} maximum degree.
We also remark that Aigner-Horev, Hefetz, and Krivelevich~\cite{AHK2025} proved that there exists an absolute constant $c > 0$ such that for any integer $1 \leq k = o(n)$, if $H$ is an $n$-vertex graph with independence number $\alpha(G) \leq cn/k$, then whp $H \cup G\left (n,\frac{8}{nk} \right )$ has Hadwiger number $\Omega\left (\frac{n}{\sqrt{\log n} \cdot \max \{\sqrt{\alpha(G)}, \log n \} \cdot k } \right )$, which is best possible up to a polylogarithmic factor in $n$ when $k=O(1)$.

Theorem~\ref{thm:fragile-main}(a) has interesting implications for long cycles and large forest minors. 
\begin{restatable}{COR}{fragilecor}\label{cor:largeforestminor}
Let $p$, $\Delta$, $H$, and $R$ be as in Theorem~\ref{thm:fragile-main}. 
Then whp $R$ contains 
\begin{enumerate}[label=\rm (\alph*)]
\item a cycle of length $\Omega\left (n^2 p/\Delta \right )$;
\item all forests on $O\left (n^2 p/\Delta \right )$ vertices as minors.
\end{enumerate}
\end{restatable}
    
As we will see later in Examples~\ref{example1}--\ref{example4}, the results in Theorem~\ref{thm:fragile-main} are best possible. But all base graphs used in those examples are connected and contain a large number of leaves. A natural question is what happens if we do not impose such properties or constraints on a base graph. 
Indeed, if a base graph has small path cover number, then we can derive tight bounds. The \emph{path cover number} of a graph~$G$, introduced by Ore~\cite{Ore1961}, is the minimum number of vertex-disjoint paths covering all vertices of $G$.
\begin{restatable}{THM}{spanningforest}\label{thm:spanningforest}
Assume that $p=p(n) \in [0,1]$ satisfies $n^2 p =\omega(1)$. Let $H:=H_n$ be a (\emph{not necessarily connected}) $n$-vertex graph 
whose path cover number is at most $n^2p/20$.
Let $G:=G(n,p)$ and $R := H \cup G$. 
Then whp the following hold:
\begin{enumerate}[label=\rm (\alph*)]
\item $\tw(R)  =  \Theta\left(\tw(H)+\min\left(n^2p, n\right)\right)$;
\item $g(R)  =  \Theta\left(g(H)+n^2p\right)$;
\item $h(R)   =  \begin{cases}
\Omega\left(h(H)+\sqrt{n^2 p}\right)& \text{if }np < \frac{11}{10}, \\
\Omega\left(h(H)+h(G)\right)&\text{otherwise.}
\end{cases}$
\end{enumerate}
\end{restatable}
We remark that 
Theorem~\ref{thm:spanningforest} would fail badly if the path cover number of $H$ is $\omega\left (n^2 p \right )$; for any arbitrarily slowly increasing function $c(n)$ with $c(n) = \omega(1)$, there exists an $n$-vertex graph $H_n$ with the path cover number $\omega\left (n^2 p \right )$ such that whp both $\tw(H_n \cup G(n,p))$ and $h(H_n \cup G(n,p))$ are at most $c(n)$. See Observation~\ref{obs:spanningforest_fail}.

\smallskip
Returning to Theorem~\ref{thm:fragile-main}, it is indeed a consequence of a more general result on fragile minor-monotone graph parameters (see Lemma~\ref{lem:main-general}). But we defer its exact statement to Section~\ref{sec:fragile-general} to avoid technical notations at this stage. An essential ingredient of the proof of Lemma~\ref{lem:main-general} is the following key lemma, which is interesting on its own and may prove useful for many other applications.

\begin{restatable}[Key lemma]{LEM}{keylemma}
\label{lem:main}
Let $C \geq 8$ be an absolute constant. Let $p = p(n) \in (0,1]$ and $\Delta = \Delta(n) \in [1,\infty)$ be the parameters satisfying $p\leq 2/n$, $n^2 p = \omega(1)$, and $\Delta \leq \frac{n^2p}{4800C}$.
Let $H:=H_n$ be an $n$-vertex connected graph with maximum degree at most~$\Delta$ and let $R:=H \cup G(n,p)$. Then whp $R$ contains vertex-disjoint \emph{connected} subgraphs $R_1 , \dots , R_m$ such that
\begin{enumerate}[label=\rm (\alph*)]
\item $96  \frac{C \Delta}{np} \leq |V(R_i)| \leq 192  \frac{C \Delta}{np}$ for each $i\in [m]$;
\item $m \geq \frac{n^2 p}{9600 C \Delta}$.
\end{enumerate}
\end{restatable}

Let us illustrate how Lemma~\ref{lem:main} can be used; this is the main idea behind Lemma~\ref{lem:main-general}, the boosting lemma. Assume that the conditions of Lemma~\ref{lem:main} hold; in particular, let $H$ be a base graph with maximum degree $\Delta$ and let $R:=H \cup G(n,p)$. Given an arbitrary minor-monotone graph parameter $f(\cdot)$, assume that we want to bound $f(R)$ from below. We begin with a standard two-round exposure approach: Take two independent random graphs $G(n,p_1)$ and $G(n,p_2)$ with $p_1 = p_2 = 1 - \sqrt{1 - p} \geq p/2$ such that $G(n,p) = G(n,p_1) \cup G(n,p_2)$. 
Applying Lemma~\ref{lem:main} we show that whp $H \cup G(n,p_1)$  contains `\emph{many}' vertex-disjoint  connected subgraphs $R_1, \dots , R_{m'}$ of comparable sizes $\Theta \left (\frac{\Delta}{np} \right )$. We contract each of $R_1, \dots , R_{m'}$ into a single vertex and delete all the other vertices; thereby we obtain a large minor $S$ of~$R$ and hence $f(R)   \geq   f(S)$. 
Then for some $m = \Theta(m') = \Theta \left (\frac{n^2 p}{\Delta} \right )$ and $q = 1 - \exp\left(-\Theta \left( \frac{\Delta^2}{n^2 p} \right) \right)$,
we show that $S$ contains a random graph $G(m,q)$ as a subgraph, due to the edges of $G(n,p_2)$ going between $R_i$ and $R_j$ for $1\le i\neq j\le m'$. Thus, it follows that 
\[ 
    f(R) \geq f(G(m,q)).
\]
What is important and interesting about this procedure is that the edge density can be `amplified' so that   $G(m,q)$ is in a \emph{supercritical} regime, e.g., with $mq>1$,  even if we would begin with a very sparse base graph $H$ (such as a path) and a \emph{subcritical} random graph $G(n,p)$, e.g., with $np<1$. 

In the proof of Theorem~\ref{thm:fragile-main}, we apply Lemma~\ref{lem:main-general} with certain choices of constants and parameters so that  $G(m,q)$ is in a supercritical regime. Then we use well-known facts that $\tw(\cdot)$, 
$g(\cdot)$, and $h(\cdot)$ are  `\emph{large}'  for $G(m,q)$ in a supercritical regime. 

The rest of the paper is organized as follows. In Section~\ref{sec:prem}, we provide definitions and facts about minor-monotone graph parameters. In Section~\ref{sec:fragile-general}, we study the boosting lemma and prove   Theorem~\ref{thm:fragile-main} by using the boosting lemma while assuming Lemma~\ref{lem:main}.
In Section~\ref{sec:spanningtrees}, we prove Theorem~\ref{thm:spanningforest}.
In Section~\ref{sec:vertex-disjoint-subgraphs}, we prove several lemmas on finding many vertex-disjoint connected subgraphs. Using these lemmas, we prove Lemma~\ref{lem:main} in Section~\ref{sec:key-connectinglem}.
In Section~\ref{sec:sharpness}, we discuss the sharpness of the results. 
Finally, in Section~\ref{sec:discussions} we discuss Lemma~\ref{lem:main} and Theorem~\ref{thm:fragile-main}(b).

\section{Preliminaries}\label{sec:prem}
Throughout this paper, every graph is simple and undirected, meaning that graphs do not have loops or parallel edges. Let $\mathbb{N}$ be the set of positive integers. For $N\in \mathbb{N}$, let $[N]:=\{1,2,\ldots,N\}$ be the set of positive integers less than or equal to~$N$.

\subsection{Definitions of tree-width, genus, and Hadwiger number}\label{sec:def-minor-monotone} 
A graph $G'$ is a \emph{minor} of a graph $G$ if a graph isomorphic to $G'$ can be obtained from $G$ by deleting vertices or edges and contracting edges. Equivalently, a graph $G'$ is a minor of a graph $G$ if there exists a collection of pairwise vertex-disjoint connected subgraphs $\{ G_u : u \in V(G') \}$ of $G$ such that $G_{u}$ and~$G_{v}$ are joined by an edge of~$G$ for each $uv \in E(G')$.

A graph parameter $f(\cdot)$ is \emph{minor-monotone} if $f(G') \leq f(G)$ 
whenever a graph $G'$ is a minor of a graph $G$.  
We will review four minor-monotone graph parameters.

Tree-width, rediscovered by Robertson and Seymour~\cite{rs1984}, is one of the most well-known and well-studied graph parameters 
in algorithmic and structural graph theory.
A \emph{tree decomposition} of a graph $G$ is a pair $(T, ({B_v})_{v \in V(T)})$ 
of a tree $T$ and 
a collection $(B_v)_{v\in V(T)}$ of subsets $B_v\subseteq V(G)$ such that the following conditions hold:
\begin{itemize}
\item $\bigcup_{v \in V(T)} B_v=V(G)$; 
\item For every edge $e = xy \in E(G)$, there exists a vertex $v$ of $T$ such that $x, y \in B_v$;
\item For every vertex $x \in V(G)$, a subset $\{ v \in V(T):x \in B_v \}$ induces a subtree of $T$.
\end{itemize}
The \emph{width} of a tree decomposition $(T, ({B_v})_{v \in V(T)})$ is $\max_{v \in V(T)} (|B_v| - 1)$. 
The \emph{tree-width} $\tw(G)$ of a graph $G$ is defined as the minimum width over all tree decompositions of~$G$.

The study of graph minors is motivated by the fact that 
for each surface $\Sigma$, 
the set of graphs that can be embedded in $\Sigma$ without crossings
is minor-closed.
The (orientable) \emph{genus} $g(G)$ of a graph $G$
is defined as the minimum number of handles to be added to the sphere 
such that $G$ can be embedded without any crossings.

We are also interested in the size of the largest complete minor in a graph.
The \emph{Hadwiger number} $h(G)$ of a graph $G$ is defined as the maximum integer $\ell \geq 0$ such that the complete graph $K_{\ell}$ is a minor of $G$.

Note that adding a new vertex or a new edge to a graph increases its tree-width and Hadwiger number by at most one. 
Adding a new edge to a graph increases its genus by at most one. 
We have $h(G) \leq \tw(G) + 1$, because $\tw(K_t) = t-1$.  

The determination of $g(K_t)$ has a long history. Heawood~\cite{Heawood1890} showed that $g(K_t) \geq \lceil (t-3)(t-4)/12 \rceil$ and conjectured that the equality holds. 
Over several decades, numerous researchers attempted to prove the conjecture and solved special cases; finally, Ringel and Youngs~\cite{ringel1968} finished the remaining cases (see~\cite{ringel1968} for the history of the conjecture).
It is shown that $g(K_t) = \left\lceil {(t-3)(t-4)}/{12}\right\rceil$ for any integer $t \geq 3$.  Because the genus is minor-monotone, it follows that 
\[
    g(G)\ge \left\lceil (h(G)-3)(h(G)-4)/12\right\rceil=\Omega\left(h(G)^2\right).
\]
In fact, various graph parameters are closely related; see an excellent survey by Harvey and Wood~\cite{HarveyWood}.

\subsection{Lower bounds for minor-closed parameters in $G(n,p)$}
We collect some known results about 
lower bounds on tree-width, 
genus, and Hadwiger number of $G(n,p)$ from~\cite{dowden2019,fountoulakis2008,lee2012,perarnau2014}, which are restated so that we can use them directly in the proof of Theorem~\ref{thm:fragile-main}.

The following theorem shows that the tree-width of $G(n,p)$ grows linearly when $p=c/n$ for a fixed $c>1$. This was initially a conjecture of Gao~\cite{Gao2012} and was proved by Lee, Lee, and Oum~\cite{lee2012} in a stronger form in terms of rank-width.
Later, Perarnau and Serra~\cite{perarnau2014} presented a direct proof.
Since both tree-width and genus do not increase by taking subgraphs, we may replace the condition $p=c/n$ with $p\geq c/n$.
\begin{THM}[Lee, Lee, and Oum~\cite{lee2012}]\label{thm:randomtw}
For any $c > 1$ and $p \geq c/n$, there exists $r=r(c) > 0$ such that whp 
$\tw(G(n,p)) \geq r\cdot n$.
\end{THM}

\begin{THM}[Dowden, Kang, and Krivelevich~\cite{dowden2019}]\label{thm:randomgenus}
For any $c > 1$ and $p \geq c/n$, there exists $r=r(c) > 0$ such that whp $g(G(n,p)) \geq r\cdot  n^2 p$.
\end{THM}

\begin{THM}[Fountoulakis, K\"uhn, and Osthus~\cite{fountoulakis2008}]\label{thm:fko} ~
\begin{enumerate}[label=\rm (\alph*)]
\item 
There exists $r>0$ such that for any $p \geq \frac{5}{4n}$, whp $h(G(n,p))\geq r  \sqrt{n}$. 
\item There exists $C' > 0$ such that for any $C'/n \leq p \leq 1/2$, whp
$h(G(n,p)) \geq \frac{n}{2 \sqrt{\log_{1/(1-p)} (np)}} $.
\item 
For any $\varepsilon > 0$ and  $\frac{1+\varepsilon}{n}\leq p \leq 1-\varepsilon$,  whp
$h(G(n,p)) = \Theta \left(n /\sqrt{\log_{1/(1-p)} (np)}\right)$.
\end{enumerate}
\end{THM}

There are many other interesting results about these minor-monotone graph parameters of $G(n,p)$; for an overview, see Table~\ref{fig:table1} in the appendix of this paper. 
\section{Perturbing a graph of small maximum degree}\label{sec:fragile-general}

In this section, we aim to prove Theorem~\ref{thm:fragile-main} assuming our key lemma, Lemma~\ref{lem:main}.
Inspired by Theorems~\ref{thm:randomtw}--\ref{thm:fko}, we focus on a lower bound of the form $f(G(n,p))\ge r \cdot \mathbb{E}(f(G(n,p)))$, leading to the following definition. 

\begin{DEFN}\label{def:f-bounded}
Let $f$ be a graph parameter, and let $\widetilde{f} : \mathbb{N} \times [0,1] \to \mathbb{R}_{\geq 0}$ be a function. 
Let $c, r >0$.
We say that $f$ is \emph{$(c,r)$-bounded from below} by $\widetilde{f}$ if for any $p:\mathbb N\to \mathbb R$ with $p(n)\in [c/n,1]$ for all large $n\in\mathbb N$,  whp
\[
f\left(G(n, p(n))\right)  \geq  r \cdot \widetilde{f}(n,p(n)).
\]
\end{DEFN}

The following \emph{boosting lemma} tells us how much a minor-monotone graph parameter can increase when we add a few edges randomly to a \emph{connected} base graph of small maximum degree. 
It will be applicable for various minor-monotone graph parameters.
Its name, the boosting lemma, comes from typical situations in which
the value of $f(G(n,p))$ is whp larger than the value of $r\cdot \mathbb{E}(f(G(m,q)))$
with $q \gg p$.
The proof uses Lemma~\ref{lem:main} and a standard two-round exposure argument. 
\begin{LEM}[Boosting lemma]
\label{lem:main-general}
Let $f$ be a minor-monotone graph parameter such that $f$ is $(c,r)$-bounded from below by~$\widetilde{f}$ for some $c>1$ and $r \geq 0$. Let $C \geq 4c$ and assume that $p = p(n) \in (0,1]$ and $\Delta = \Delta(n)\in [1,\infty)$ satisfy $p\leq 2/n$ and $n^2 p = \omega(1)$, and $\Delta \leq \frac{n^2 p}{9600C}$. Let    
\[ 
    m= m(n) := \left\lceil \frac{n^2 p}{19200C\Delta }\right\rceil , ~~
    M = M(n) := \frac{(96 C \Delta)^2}{n^{2} p},~~
    \text{and}~~ q=q(n) := 1 - e^{-M}.
\]
Let $H$ be an $n$-vertex connected graph with maximum degree at most $\Delta$, and let $R:=H\cup G(n,p)$. Then whp 
\[f(R)  \geq    
\begin{cases}
r \cdot \widetilde{f} (m,q) & \text{if }\Delta \leq \sqrt{n^2 p \log \left (n^2 p \right ),} \\
f (K_{m}  ) & \text{otherwise.}
\end{cases}
\]
\end{LEM}

\begin{proof}[Proof of Lemma~\ref{lem:main-general}, assuming Lemma~\ref{lem:main}]  
First we take two independent random graphs $G(n,p_1)$ and $G(n,p_2)$  where $1 - p = (1 - p_1)(1 - p_2)$ with $p_1 = p_2 := 1- \sqrt{1-p} \geq p/2$. Then $G(n,p)$ has  the same probability distribution as $G(n,p_1) \cup G(n,p_2)$. 
Let $R_0 := H \cup G(n,p_1)$. 
Then $R$ has the same probability distribution as $R_0 \cup G(n,p_2)$, so with slight abuse of notation we write  $R=R_0 \cup G(n,p_2)$.

By Lemma~\ref{lem:main}, whp $R_0$ contains \emph{vertex-disjoint connected} subgraphs $R_1 , \dots , R_{m'}$ such that
\begin{enumerate}[label=\rm (\Alph*)]
\item $96 C \Delta (np_1)^{-1} \leq |V(R_i)| \leq 192 C \Delta (np_1)^{-1}$ for each $i\in [{m'}]$;  
\item ${m'} \geq {n^2 p_1}/ (9600 C\Delta)$.
\end{enumerate}
Next, consider a graph $S$ such that $V(S) = [{m'}]$ and for $1 \leq i < j \leq m'$, $ij \in E(S)$ if and only if there exists an edge of $G(n,p_2)$ 
having one end in $R_i$ and another end in $R_j$.
Observe that $S$ is isomorphic to a minor of $R$, since if we contract $R_i$ into a single vertex $s_i$ for each $i\in [{m'}]$, the resulting minor of $R$ has a subgraph isomorphic to $S$.
Furthermore, for each $1\le i<j \le m'$, the probability that an edge of $G(n,p_2)$ exists between $R_i$ and $R_j$ is
\begin{align*}
1 - (1 - p_2)^{|V(R_i)||V(R_j)|}  
&  \geq
1- (1 - p_2)^ { \left(\frac{96C\Delta}{np_1} \right)^2}, \\
\intertext{and by the inequalities $1+x\le e^{x}$ and $p_1=1-\sqrt{1-p}\le p$, we have }
& \geq
1 - \exp\left(-  p_2 \frac{(96 C \Delta)^2}{n^2 p_1^2}\right)
& \ge 1 - \exp\left(-  \frac{(96 C \Delta)^2}{n^2 p}\right)  
=1- e^{-M} = q.
\end{align*}
Since $p_1\ge p/2$, we observe that 
\[ 
    {m'}  \geq \left\lceil \frac{n^2 p}{19200 C \Delta}\right\rceil  = m.
\]  
Hence, $S$ contains a random graph $G(m, q)$ as a subgraph.  
Since $f$ is minor-monotone, whp 
\begin{align}
f(R) \geq f(S) \geq f(G(m, q)). \label{eq:f-minor-subgraph}
\end{align}

If $\Delta \leq \sqrt{n^2 p \log \left (n^2 p \right )}$, then 
\[ 
    m\ge  \frac{n^2 p}{19200 C \Delta} =\omega(1),
    \qquad 
    q=1 - \exp\left({-\frac{(96 C \Delta)^2}{n^2 p} }\right).
\] 
If $\frac{(96C\Delta)^2}{n^2p}\ge 1$, then $q\ge 1-e^{-1}$ and therefore $mq=\omega(1)>c$ for all large $n$.
If $\frac{(96C\Delta)^2}{n^2p}< 1$, then as $1-e^{-x} \geq (1-e^{-1})x$ for all $0<x<1$, we have 
\[ mq> m(1-e^{-1})\frac{(96 C \Delta)^2}{n^2 p}\ge (1-e^{-1})\frac{96^2 C\Delta}{19200}\ge (1-e^{-1})\frac{96^2 C}{19200}>c.\] 
So in both cases, $q>c/m$ for all large $n$.
Since $f$ is $(c,r)$-bounded below by $\widetilde{f}$ and $q>c/m$, by Definition~\ref{def:f-bounded} whp  
\[f(G(m, q))  \geq r  \cdot \widetilde{f}  (m, q).\]

On the other hand, if $\Delta \geq \sqrt{n^2 p \log \left (n^2 p \right )}$, then $1-q  \leq \left(n^2 p\right)^{-(96C)^2} $ for all large $n$. Since $m^2(1-q) \leq \left (n^2 p \right )^{2-(96C)^2}   = o(1)$, whp the random graph $G(m, q)$ is isomorphic to~$K_m$
and by \eqref{eq:f-minor-subgraph}, whp 
\[ 
    f(G(m,q)) \geq   f(K_m).
\] 

These two inequalities with  \eqref{eq:f-minor-subgraph} imply the desired conclusion.
\end{proof}

Now it is straightforward to obtain Theorem~\ref{thm:fragile-main} from the boosting lemma.

\fragilemain*
\begin{proof}
    (a)
    Since $\tw(R)\ge \tw(H)$, it is enough to prove that $\tw(R)\ge \Omega\left(n^2p/\Delta\right)$.
    Let $\widetilde{t}(x,y) := x$.
    By Theorem~\ref{thm:randomtw}, there exists $r>0$ such that   $\tw$ is $(5/4,r)$-bounded from below by $\widetilde{t}$.

    If $p>2/n$, then whp $\tw(R)\ge \tw(G(n,p))\ge r n$. 
    Thus, we may assume that $p\le 2/n$.
    By  Lemma~\ref{lem:main-general} with $c = 5/4$,  $C = 5$, and 
    $m := \left \lceil \frac{n^2 p}{19200C\Delta}\right \rceil  = \Omega\left(\frac{n^2 p}{\Delta}\right)$,  whp  
\[\tw(R)  \geq  
\begin{cases}
r \cdot \widetilde{t} (m,q) = r \cdot m = \Omega\left(\frac{n^2 p}{ \Delta}\right)& \text{if }\Delta \leq \sqrt{n^2 p \log \left (n^2 p \right )}, \\
\tw (K_{m}  ) = m-1 =\Omega\left(\frac{n^2 p}{\Delta}\right) & \text{otherwise.}
\end{cases}
\]

\bigskip \noindent 
(b) 
Since $g(R)\ge g(H)$, it suffices to show that $g(R)\ge \Omega\left(\min\left(n^2p,\left(n^2p/\Delta\right)^2\right)\right)$.

If $p \geq 2/n$, then by Theorem~\ref{thm:randomgenus}, there exists $r>0$ such that $g(G(n,p)) \geq r\cdot n^2 p$, and thus $g(R) =\Omega\left(n^2 p\right)$.

Now we assume that $p \leq 2/n$. 
Let $\widetilde{g}(x,y) := x^2 y$.
By Theorem~\ref{thm:randomgenus}, there exists $r>0$ such that the genus $g$ is $(5/4,r)$-bounded from below by $\widetilde{g}$. Applying Lemma~\ref{lem:main-general} with  $c = 5/4$,  $C = 5$, $m := \left \lceil \frac{n^2 p}{19200C\Delta} \right \rceil  = \Omega \left (\frac{n^2 p}{\Delta} \right )$,
$M=(96C\Delta)^2/\left (n^2p \right )$,
and $q:=1-e^{-M}$, we deduce that whp
\[g(R)   \geq  
\begin{cases}
r \cdot \widetilde{g} (m,q) = r \cdot m^2 q =\Omega\left(\left(\frac{n^2p}{\Delta}\right)^2 q\right) & \text{if }\Delta \leq \sqrt{n^2 p \log \left (n^2 p \right )}, \\
g(K_{m})  = \Omega\left (m^2 \right ) = \Omega\left(\left(\frac{n^2 p}{\Delta}\right)^2\right)  & \text{otherwise.}
\end{cases}
\]

Furthermore, observing $q\ge (1-e^{-1}) \min\left(1,  \frac{(96C\Delta)^2}{n^2p} \right)$, we complete the proof.

\bigskip \noindent 
  
(c) 
Since $h(R)\ge h(H)$, it suffices to show that $h(R)=\Omega\left(\min\Bigl( \sqrt{\frac{n^2p}{\log{\Delta}}}, \frac{n^2p}{\Delta \sqrt{\log \Delta}}  \Bigr)\right)$.
We choose  $r$, $C'$ so that Theorem~\ref{thm:fko}(a)--(b) hold and define 
\begin{equation}
\widetilde{h}(n,p) :=
\begin{cases}
r \sqrt{n}& \text{if } \frac{5}{4n} \leq p < \frac{C'}{n}, \\
\frac{n}{2\sqrt{\log_{1/(1-p)} (np)}} & \text{if } \frac{C'}{n} \leq p \leq \frac{1}{2},\\
\frac{n}{2\sqrt{\log_{2} n}} &\text{if } \frac{1}{2} < p \leq 1.
\end{cases} \label{eqn:widetilde h}
\end{equation}

Then $h(G(n,p)) \geq \widetilde{h}(n,p)$ for $\frac{5}{4n} \leq p \leq 1/2$. For $p \geq 1/2$, by Theorem~\ref{thm:fko}(b), 
we have that whp $h(G(n,p)) \geq h(G(n,1/2)) \geq \frac{n}{2\sqrt{\log_2{n}}} = \widetilde{h}(n,p)$.
Thus, $h$ is $(5/4,1)$-bounded from below by~$\widetilde{h}$.

Let $C=5$, 
$M=\frac{(96C\Delta)^2}{n^2p}$,  $m= \left \lceil \frac{n^2p}{19200C\Delta} \right \rceil $, and $q = 1- e^{-M}$.
By Lemma~\ref{lem:main-general}, we conclude that
 \begin{align}
 h(R) \geq 
\begin{cases}
\widetilde{h}(m,q) & \text{if } \Delta \leq  \sqrt{n^2 p \log \left (n^2 p \right )}, \label{eqn:h-widetilde h}\\
h(K_{m}) = m & \text{otherwise}. 
\end{cases}
 \end{align}

If $\Delta > \sqrt{n^2 p \log \left (n^2 p \right )}$, then by~\eqref{eqn:widetilde h} and~\eqref{eqn:h-widetilde h}, we have that whp $h(R)\geq m = \Omega\left ({n^2p}/{\Delta} \right )$.
Thus we may assume that $\Delta \le \sqrt{n^2 p \log \left (n^2 p \right )}$.

If $q\leq 1/2$, then 
$M\le \log 2$,
$\Delta < \sqrt{n^2 p}$, and 
$\frac{M}{2\log 2}\leq q\leq M$,
since $1 - e^{-x} - \frac{x}{2 \log 2} \geq 0$ for any $x \in [0, \log 2]$.
If $q \leq C'/m$, then 
$C' \geq mq \geq \frac{mM}{2\log 2} \ge \frac{96^2 C\Delta}{19200 \cdot 2\log 2} > 1.5\Delta $. In this case, 
\eqref{eqn:widetilde h} and~\eqref{eqn:h-widetilde h} imply that whp  
\[ h(R)\geq \widetilde{h}(m,q) = r\sqrt{m} \ge 
r\sqrt{\frac{n^2p}{19200C\Delta}}
> r\sqrt{\frac{1.5 n^2p}{19200C C'}}
=\Omega\left(
\sqrt{\frac{n^2p}{\log\Delta}}\right).\] 

If $C'/m < q \leq 1/2$, then again by~\eqref{eqn:widetilde h} and~\eqref{eqn:h-widetilde h}, we have that whp 
\[h(R)\geq \widetilde{h}(m,q)=\frac{m}{2\sqrt{\log_{1/(1-q)} (mq)}} =\Omega\left(\sqrt{\frac{n^2p}{\log\Delta}}\right),\]
where the last equality follows since
\begin{align*}
    \frac{m}{\sqrt{\log_{1/(1-q)} (mq)}} = \frac{m}{ \sqrt{-\log (mq) / \log (1-q)}} = \frac{m \sqrt{M}}{\sqrt{\log(mq)}},
\end{align*}
and $\frac{m \sqrt{M}}{\sqrt{\log(mq)}} = \Theta \left ( \sqrt{ \frac{n^2 p}{\log \Delta}}\right )$ as we have $m = \Theta \left (n^2 p / \Delta \right )$, $M = \Theta \left ( \Delta^2 / n^2 p \right )$, and $q = 1 - e^{-M} = \Theta \left (\Delta^2 / n^2 p \right )$.

Now, assume that $q>1/2$. Then we have $\Delta \geq \frac{\sqrt{n^2 p \log 2}}{96C}$
and $\Delta \ge \frac{1}{\Delta}\left(\frac{\sqrt{n^2 p \log 2}}{96C}\right)^2 = \Omega(m)$.
In this case, by~\eqref{eqn:widetilde h} and~\eqref{eqn:h-widetilde h}, we have that whp 
\[ h(R) \geq h(G(m,q)) \geq \widetilde{h}(m,q) = \frac{m}{2\sqrt{\log_2{m}}} = \Omega\left(\frac{n^2p}{\Delta \sqrt{\log{\Delta}} }\right).\]
This completes the proof.
\end{proof}

\begin{THM}[Birmele~\cite{birmele2003}]\label{thm:tw-cycle}
Let $k \geq 2$. Every graph with tree-width at least $k$ has a cycle of length more than $k$.
\end{THM}

A \emph{path-width} of a graph $G$ is defined as the minimum width of a tree decomposition of $G$ where the underlying tree of the decomposition is a path. Note that the path-width of a graph is larger than or equal to its tree-width.

\begin{THM}[Bienstock, Robertson, and Seymour~\cite{bienstock1991}]\label{thm:pw-forest}
Let $k \geq 1$ be an integer. 
Every graph with path-width at least $k-1$ has all $k$-vertex forests as minors.
\end{THM}

Then the following Corollary~\ref{cor:largeforestminor} follows immediately from Theorems~\ref{thm:fragile-main}(a),~\ref{thm:tw-cycle}, and~\ref{thm:pw-forest}.

\fragilecor*

\section{Perturbing a graph of small path cover number}\label{sec:spanningtrees}

Recall that a path cover number of a graph $G$ is the minimum number of vertex-disjoint paths covering all vertices of $G$.
In this section,  we are going to prove Theorem~\ref{thm:spanningforest}. Let us first prove a useful lemma to find many disjoint paths.

\begin{LEM}\label{lem:manypath}
    Let $k,n,t \in \mathbb{N}$.
    If $G$ is a graph on $n$ vertices whose path cover number is at most $t$, 
    then 
    it contains at least $n/k-t$ vertex-disjoint paths each having exactly $k$ vertices.
\end{LEM}
\begin{proof}
    Let $G$ be an $n$-vertex graph and 
    let $P_1, P_2 , \ldots, P_t$ be vertex-disjoint paths covering all vertices of~$G$. For each path $P_i$, we choose as many vertex-disjoint subpaths on exactly $k$ vertices as possible.
    Then all but $t(k-1)$ vertices of $V(G)$ can be covered by vertex-disjoint paths on $k$ vertices, hence there are at least 
    \[ 
    \frac{ n- t(k-1) }{k} \ge \frac{n}{k}-t
    \] 
    vertex-disjoint paths on exactly $k$ vertices.
\end{proof}

Let us restate Theorem~\ref{thm:spanningforest} and present its proof.
\spanningforest*
\begin{proof}
    Suppose that $p\ge \frac{11}{10n}$.
By Theorems~\ref{thm:randomtw} 
and \ref{thm:randomgenus}, whp $\tw(G) \ge \Omega(n)$ 
and $g(G) \ge \Omega \left (n^2 p \right )$.  
Then (a) holds because 
the tree-width of an $n$-vertex graph is at most $n-1$
and $\tw(R)\ge (\tw(H)+\tw(G))/2$.
Since adding one edge to a graph can increase the genus by at most~$1$, 
we have $(g(H)+g(G))/2 \le g(R)\le g(H)+ O\left (n^2p \right )$ whp, proving (b).
Observe that $h(R)\ge (h(H)+h(G))/2$ and therefore (c) holds. 

Hence, we may assume that $p < \frac{11}{10n}$. 
    We first prove the upper bounds in (a) and (b).
Since adding one edge to a graph can increase its genus and tree-width 
by at most~$1$, 
we have $\tw(R)\le \tw(H)+O\left (n^2p \right )$ and $g(R)\le g(H)+ O\left (n^2p \right )$ whp
and trivially $\tw(R)\le n-1$, 
proving upper bounds for (a) and (b).

For lower bounds, as $\tw(R)\ge \tw(H)$, $g(R)\ge g(H)$, and $h(R)\ge h(H)$, 
it suffices to show that 
\begin{align*}
\tw(R) & \ge \Omega\left (\min \left (n^2 p , n \right ) \right ),\\
g(R) &\ge \Omega \left (n^2 p \right ), \\
h(R) &\ge  \Omega\left (\sqrt{n^2 p} \right ).
\end{align*}

Let $k$ be an integer such that $\frac{29}{10np} \leq k \leq \frac{4}{np}$, 
and let $m := \left \lceil \frac{n}{k} - \frac{n^2p}{20} \right \rceil $. 
Observe that such an integer~$k$ exists because $p<\frac{11}{10n}$. 
Furthermore, $m \geq  \frac{n^2 p}{5}$.  By Lemma~\ref{lem:manypath}, 
$H$ has at least $m$ vertex-disjoint paths $P_1 , \dots , P_m$, each having exactly $k$ vertices. 
For each $1 \leq i<j \leq m$, 
the probability that $R$ has an edge between $P_i$ and $P_j$ is
\[ q := 1 - (1-p)^{k^2} \geq k^2p - \frac{k^2 (k^2-1)}{2} p^2\ge pk^2 \left(1- \frac{pk^2}{2}\right)\ge  pk^2 \left( 1 - \frac{8}{n^2p}\right).\]
Let $G':=G/E(P_1)/E(P_2)/\cdots / E(P_m)$, which is a graph obtained from $G$
by contracting edges in $\bigcup_{i=1}^m E(P_i)$.
Then $G'$ is a minor of $R$ and $G'$ contains $G(m,q)$ as a subgraph. 
Thus $f(R)\geq f(G(m,q))$ for any minor-monotone graph parameter $f(\cdot)$. 

Since  $n^2p=\omega(1)$ and $m \geq \frac{n^2 p}{5}$, for all sufficiently large $n$, we have  
\[ m q \geq \frac{n^2 p}{5} pk^2 \left( 1 - \frac{8}{n^2p}\right)
\geq \frac{n^2 p}{5} \cdot \left(\frac{29}{10}\right)^2  \frac{1}{n^2 p}
\left( 1 - \frac{8}{n^2p}\right) > \frac{5}{4}.\]
Hence, by Theorems~\ref{thm:randomtw}--\ref{thm:fko},  
whp 
\begin{align*}
\tw(R) &\geq \tw(G(m,q)) \ge  \Omega(m) \ge \Omega \left (n^2 p \right ), \\
g(R) &\geq g(G(m,q)) \ge  \Omega \left (m^2 q \right ) \ge \Omega \left (n^2 p \right ), \\
h(R) &\geq h(G(m,q)) \ge \Omega\left (\sqrt{m} \right ) \ge \Omega \left (\sqrt{n^2 p} \right ), 
\end{align*}
as desired.
\end{proof}

\section{Finding many large vertex-disjoint connected subgraphs}\label{sec:vertex-disjoint-subgraphs}
We begin with a lemma based on the following classical result of Ajtai, Koml\'os, and Szemer\'edi~\cite{ajtai1981} and Fernandez de la Vega~\cite{Fernandezdelavega1979}. 

\begin{LEM}\label{lem:longpath}
For every $n\in \mathbb{N}$, the random graph $G(n , 20/n)$ contains a path on at least $n / 5$ vertices with probability at least $1 - e^{-n}$.
\end{LEM}
We present a proof that is based on the idea of the proof in~\cite[Theorem 3.4 in Chapter 1]{random2016}.
We will use the following lemma.
\begin{LEM}[Ben-Eliezer, Krivelevich, and Sudakov~{\cite[Lemma 4.4]{BKS2012}}]\label{lem:pseudorandom}
    Let $n,k\in \mathbb{N}$ and $G$ be an $n$-vertex graph such that for every pair $(X,Y)$ of disjoint sets of $k$ vertices, there is at least one edge joining $X$ and $Y$. Then $G$ has a path whose length is at least $n-2k+1$.
\end{LEM}
\begin{proof}[Proof of Lemma~\ref{lem:longpath}]
We may assume that $n>5$.
Let $G := G(n, 20/n)$ and let $k := \lceil 2n/5 \rceil$. 
Then $k\le 2n/5 + 1$ and $n - 2k + 2 \geq n/5$. By the union bound, the probability that there exist two disjoint subsets $X,Y \subseteq V(G)$ of size $k$ with no edges between $X$ and $Y$ is at most
\begin{align*}
    \binom{n}{k} \binom{n-k}{k} \left(1 - \frac{20}{n}\right)^{k^2} 
    &\leq \left(\frac{en}{k}\right)^{2k} \cdot e^{-\frac{20k^2}{n}} \\
    &  = \left ( \frac{en}{k} \cdot e^{-\frac{10k}{n}} \right )^{2k}
    \leq \left( \frac{5e}{2} e^{-4}\right)^{2k}
    \leq \left(\frac{5}{2e^3} \right)^{\frac{4n}{5}}
    < e^{-n}.
\end{align*}
Thus, with probability at least $1 - e^{-n}$, there exists a path on at least $n/5$ vertices in $G$ by Lemma~\ref{lem:pseudorandom}.
\end{proof}
Now we are going to prove lemmas, which allow us to find many large vertex-disjoint connected subgraphs of $G(n,p)$.

\begin{LEM}\label{lem:vertex-disj-conn-1}
Let $p=p(n)\in (0,1]$, let $u>0$ be a real number, and let $G := G(n,p)$.
If $\mathcal C = \{ X_1 , \dots , X_m \}$ is a nonempty set of disjoint subsets of $V(G)$ such that $\abs{X_i} \geq u$ for all $1 \leq i \leq m$, then with probability at least $1 - m \exp\left(-pu^2 \right)$, $G$ has an edge between $X_i$ and $X_{i+1}$ for all $1\le i<m$.
\end{LEM}
\begin{proof}
For each $j\in [m-1]$, the probability that there are no edges of $G$ between $X_j$ and $X_{j+1}$  is 
    \[ 
     (1-p)^{\abs{X_j}\abs{X_{j+1}}}
    \le (1-p)^{u^2}
    \le \exp \left (-pu^2 \right ).
    \] 
Therefore the probability that there are no edges of~$G$ between $X_j$ and $X_{j+1}$ for some $j\in \{1,2,\ldots,m-1\}$ is at most $m \exp \left (- pu^2 \right )$ by the union bound.
\end{proof}

\begin{LEM}\label{lem:vertex-disj-conn-2}
    Let $p=p(n)\in (0,1]$, let $u>0$ be a real number, and let $G:=G(n,p)$.
    If $\mathcal C$ is a nonempty set of disjoint subsets of $V(G)$ such that
    \begin{enumerate}[label=\rm (\alph*)]
        \item $\abs{X} \geq u$ for all $X\in\mathcal C$, and
        
        \item $\abs{\mathcal C} \geq \max\left (40,\:40/\left(pu^2\right) \right)$,
    \end{enumerate}
    then with probability at least $1 - e^{-\abs{\mathcal{C}}}$, the set~$\mathcal C$ contains 
     $m = \lceil \abs{\mathcal C}/5 \rceil$ distinct sets $X_1,X_2,\ldots,X_m$
    such that $G$ has an edge between $X_i$ and $X_{i+1}$ for every $1\le i<m$.
\end{LEM}
\begin{proof}
    Let $q := 1 - (1 - p)^{u^2}$, and let $H$ be an auxiliary random graph such that $V(H) = \mathcal{C}$ and $\{ X, Y \} \in E(H)$ if and only if there exists an edge of $G$ between $X$ and $Y$ for distinct $X,Y \in \mathcal{C}$. Then, the probability that there are no edges of $G$ between $X$ and $Y$ is 
    \[ 
     (1-p)^{\abs{X}\abs{Y}}
    \le (1-p)^{u^2} = 1 - q
    \le \exp \left (-pu^2 \right ),
    \]
    so we can regard $G(\abs{\mathcal{C}},q)$ as a subgraph of $H$, coupling $G(\abs{\mathcal{C}},q)$ and $H$.
    Now we claim that $q \geq 20/\abs{\mathcal{C}}$. 
    To see this, if $pu^2 \geq \log 2$ then $q \geq 1 - \exp\left (-pu^2 \right ) \geq 1/2 \ge 20/\abs{\mathcal{C}}$, as $\abs{\mathcal{C}} \geq 40$. 
    Otherwise, if $pu^2 < \log 2$, then $q \geq 1 - \exp\left(-pu^2 \right) \geq pu^2 / 2 \geq 20/\abs{\mathcal{C}}$, as $\abs{\mathcal{C}} \geq 40/\left(pu^2 \right)$.
    
    Thus, by Lemma~\ref{lem:longpath}, 
    with probability at least $1 - e^{-\abs{\mathcal{C}}}$, $H$ has a path on at least $\abs{\mathcal{C}}/5$ vertices. Equivalently, with probability at least $1 - e^{-\abs{\mathcal{C}}}$, 
    there exist $m \geq \abs{\mathcal{C}}/5$ distinct sets $X_1 , \dots , X_m \in \mathcal{C}$ such that $G$ has an edge between $X_i$ and $X_{i+1}$ for every $1 \leq i < m$, as desired.
\end{proof}

\begin{LEM}\label{lem:decomp_conn}
    Let $k$ be a positive real.
    Let $G$ be a graph 
    and let $T_1 , T_2 , \ldots , T_m$ be vertex-disjoint connected subgraphs of $G$ such that 
    $V(G)=\bigcup_{i=1}^m V(T_i)$ and
    $\abs{V(T_i)}\le k$ for all $i\in [m]$.
    If $G$ has an edge from $T_i$ to $T_{i+1}$ for every $i\in [m-1]$, then $G$ contains vertex-disjoint connected subgraphs $R_1 , R_2 , \ldots , R_{m'}$ such that 
    $k\le \abs{V(R_i)}< 2k$ for all $i\in [m']$  and 
    $\abs{V(G)}-\sum_{i=1}^{m'}\abs{V(R_i)}< k$.
\end{LEM}
\begin{proof}
    We proceed by induction on $\abs{V(G)}$. The statement is trivial if $\abs{V(G)}<k$. Thus, we may assume that $\abs{V(G)}\ge k$ and therefore $m\ge 1$.
    Let $i$ be the minimum positive integer such that $\sum_{j=1}^i \abs{V(T_j)}\ge k$. By the assumption, $\sum_{j=1}^i \abs{V(T_j)}< 2k$. Let $R_1=T_1\cup T_2\cup \cdots \cup T_i$.
    We now apply the induction hypothesis to $G-\bigcup_{j=1}^i V(T_j)$ to find the remaining $R_2,R_3,\ldots,R_{m'}$.
\end{proof}

\section{Proof of the Key Lemma}\label{sec:key-connectinglem}

In order to prove Lemma~\ref{lem:main}, we will first 
present a lemma which allows us to obtain many vertex-disjoint connected subgraphs of $V(H\cup G(n,p))$ of size between $\Omega(1/(np))$ and $O(\Delta / (np))$, which cover almost all vertices in $V(H\cup G(n,p))$.  
A similar lemma with essentially the same proof appears in several papers; for example, \cite[Proposition 4.5]{krivelevich2006} and \cite[Lemma 3]{boettcher2019}.

\begin{LEM}[Tree partitioning lemma]\label{lem:frag}
    Let $\ell > 2$, $\Delta\ge 1$ be reals.
    Every connected graph $G$ with the maximum degree at most~$\Delta$
    admits a partition $V_0, V_1, \ldots, V_s$ of its vertex set with $s\ge 0$ such that 
    \begin{enumerate}[label=\rm (\alph*)]
        \item $0\le \abs{V_0}< \ell$, 
        \item $G[V_i]$ is connected for all $0\le i\le s$, 
        \item $\ell \le \abs{V_i}< \ell\Delta $ for all $1\le i\le s$.
    \end{enumerate}
\end{LEM}
\begin{proof}
    We proceed by induction on $\abs{V(G)}$. We may assume that $G$ is a tree by taking a spanning tree.
    We may assume that $\abs{V(G)}\ge \ell$ (thus $\Delta \geq 2$) because otherwise, we can take $V_0:=V(G)$. 
    Fix a vertex of degree $1$ as a root. 
    For each vertex $v$ of $G$, let $T_v$ be the subtree of $G$ rooted at $v$ consisting of $v$ and all of its descendants. Let $w$ be a vertex such that $\abs{V(T_w)}\ge \ell$ and the distance from the root to $w$ is maximum.
    Then $w$ has at most $\Delta-1$ children. (Note that this is still true even if $w$ is the root because $1\le \Delta-1$.)
    By the maximality of the distance from the root to~$w$, for every child $u$ of~$w$, we have $\abs{V(T_u)}<\ell$ and therefore $\abs{V(T_w)}<(\Delta-1)\ell+1\le \ell\Delta$. 
    Now $G':=G-V(T_w)$ is connected and we apply the induction hypothesis to $G'$ to find a partition $V_0,V_1,\ldots,V_s$ of $V(G')$. We take $V_{s+1}:=V(T_w)$.
\end{proof}

Now we restate our key lemma to be proved in this section.
\keylemma*

We first sketch our proof strategy. 
First, by applying Lemma~\ref{lem:frag} with $\ell:=96\frac{C}{np}$, we obtain many vertex-disjoint connected subgraphs of~$H$, called \emph{clusters}, each having at least $\ell$ vertices and at most $\ell \Delta$ vertices. 
Our goal is to merge them using edges of $G(n,p)$ to obtain many vertex-disjoint connected subgraphs, each having a similar number of vertices required by~(a).
To this end, we will conduct the following four steps.
\begin{enumerate}[label=\rm S\arabic*]
\item (Dyadic decomposition).  
We partition the set of clusters into \emph{levels} such that the number of vertices in each cluster is within a factor $2$ of each other in the same level.
\item (Connecting clusters in each level).  
For each level in the dyadic decomposition, whp we can find a linear ordering of many clusters so that $G(n,p)$ has an edge between consecutive clusters;  see Claim~\ref{claim:longpath}.  
\item (Connecting cluster between consecutive levels). 
Discarding some clusters in each linear ordering, whp we can concatenate all linear orderings into a single linear ordering, where the union of clusters in the linear ordering contains $\Omega(n)$ vertices and $G(n,p)$ has an edge between consecutive clusters in the ordering; see Claim~\ref{claim:merge}.
\item (Merging consecutive clusters into connected subgraphs). 
We merge consecutive clusters into connected subgraphs on $\Theta(\Delta / (np))$ vertices to obtain $\Theta \left (n^2 p / \Delta \right )$ vertex-disjoint connected subgraphs.
\end{enumerate}

\smallskip 
\begin{proof}[Proof of Lemma~\ref{lem:main}]
Let $n$ be sufficiently large and $\ell :=96 \frac{C}{np}$. 
As $\ell> 2$, 
by Lemma~\ref{lem:frag}, there is a collection $\mathcal{F}$ of disjoint sets $X_1 , \dots , X_s \subseteq V(H)$ satisfying the following. 
\begin{enumerate}[label=\rm (\Alph*)]
\item $\sum_{i=1}^{s} |X_i| > |V(H)| - 96\frac{C}{ np}$;
\item $H[X_i]$ is connected for each $i\in [s]$;
\item $96\frac{C}{np} \leq |X_i| < 96\frac{C\Delta}{np}$ for each $i\in [s]$. 
\end{enumerate}
We call each $X_i$ a \emph{cluster}. 
For $1 \leq i \leq \lceil \log_2 \Delta \rceil$, let 
\[
\mathcal{V}_i := \{ S  \in \mathcal{F}\colon 2^{i-1} \ell \leq |S| < 2^{i} \ell \},\quad V_i := \bigcup_{S \in \mathcal{V}_i} S,\quad u_i := 2^{i-1}\ell ,\quad n_i := |\mathcal{V}_i|,
\]
and
\begin{equation}\label{eqn:c_i}
c_i :=   
\max \left (\frac{80}{ pu_i } \:, \: \frac{n}{50 \log_2 \left (n^2 p \right )} \right )= 
\max \left(\frac{5n}{6\cdot 2^{i-1}C} , \frac{n}{50\log_2\left (n^2p \right )}
\right).
\end{equation}
Then, $\bigcup_{i=1}^{\lceil \log_2\Delta\rceil} \mathcal V_i = \mathcal F$.

We define $ \mathcal{A} := \{ i: 1 \leq i \leq \lceil \log_2 \Delta \rceil , ~ \abs{V_i} \geq c_i \} $ and call $i\in  \mathcal{A}$ a \emph{level}. Then there are at least $n/2$ vertices in $\bigcup_{i \in \mathcal{A}} V_i$, because
\begin{equation}\label{eqn:size}
\sum_{i \notin \mathcal{A}} \abs{V_i} \leq  
\sum_{i=1}^{\lceil \log_2 \Delta \rceil} \frac{5n}{6\cdot 2^{i-1}C} 
+\frac{n}{50 \log_2 \left (n^2 p \right )} \cdot \lceil \log_2 \Delta \rceil \leq
\frac{5n}{3C}+\frac{n}{50}
\leq \frac{5n}{12}, 
\end{equation}
where we use $C \geq 8$ for the last two inequalities.
Since each element of $\mathcal V_i$ has less than $2u_i$ vertices, for each $i\in \mathcal A$,  
\[ 
n_i \geq \frac{\abs{V_i}}{2u_i}>\frac{c_i}{2u_i} 
> \frac{40}{pu_i^2}
\ge \frac{40}{p(\Delta\ell)^2}
=\frac{40n^2p}{96^2 C^2\Delta^2}. 
\]

Let $G := G(n,p)$.

\begin{CLAIM}\label{claim:longpath}
    Whp, for every $i \in \mathcal{A}$, there exist $m_i = \lceil n_i / 5 \rceil$ distinct sets $S_{i,1} , \dots , S_{i,m_i} \in \mathcal{V}_i$ such that there is an edge of $G$ between $S_{i,j}$ and $S_{i,j+1}$ for every $j \in [m_i - 1]$.
\end{CLAIM}
\begin{subproof}[Proof of Claim 1]
    If $u_i \ge (n/p)^{1/3}$, then $n_i = \abs{\mathcal{V}_i} \leq n/u_i \leq \left (n^2 p \right )^{1/3}$ and $pu_i^2 \ge \left (n^2 p \right )^{1/3}$. 
    Then by Lemma~\ref{lem:vertex-disj-conn-1}, 
    with probability at least $1 - \left (n^2 p \right )^{1/3} e^{-\left (n^2 p \right )^{1/3}}$, $G$ has an edge between $S_{i,j}$ and $S_{i,j+1}$ for every $1 \leq j < n_i$, where $S_{i,1} , \dots , S_{i,n_i}$ is a fixed ordering of the subsets in~$\mathcal{V}_i$.

    If $u_i < (n/p)^{1/3}$, then 
    $n_i = \abs{\mathcal{V}_i} >\frac{\abs{V_i}}{2 u_i} \geq \frac{c_i}{2u_i}$ 
    because $\abs{S}< 2u_i$ for every $S \in \mathcal{V}_i$. 
    Since $c_i\ge 80/(pu_i)$, we have $n_i\ge 40/\left (pu_i^2 \right )$.
    In addition, as $c_i\ge \frac{n}{50\log_2\left (n^2p \right )}$, we have \[n_i\geq \frac{n}{100 u_i \log_2 \left (n^2 p \right )} > \frac{\left (n^2 p \right )^{1/3}}{100 \log_2 \left (n^2 p \right )} > 40 \] 
    because $n^2 p=\omega(1)$.

Thus, by Lemma~\ref{lem:vertex-disj-conn-2}, with probability at least $1 - e^{-n_i} \geq 1- \exp \left (-\frac{\left (n^2 p \right )^{1/3}}{100\log_2 \left (n^2p \right )} \right )$, there exist $m_i = \lceil n_i / 5 \rceil$ distinct sets $S_{i,1} , \dots , S_{i, m_i} \in \mathcal{V}_i$ such that there exists an edge of $G$ between $S_{i,j}$ and $S_{i,j+1}$ for every $j \in [m_i - 1]$.

Thus, since $\abs{\mathcal A} \leq \lceil \log_2 \Delta \rceil \leq \log_2 \left (n^2 p \right )$, 
the claim follows with probability at least \[ 1 - \abs{\mathcal A}\max\left(\left (n^2 p \right )^{1/3} e^{- \left (n^2 p \right )^{1/3}}, e^{-\frac{\left (n^2 p \right )^{1/3}}{100\log_2 \left (n^2p \right )}} \right) = 1 - o(1)\] by the union bound.
\end{subproof}

We now fix the edges of $G[V_i]$ for all $i \in \mathcal{A}$ given by   Claim~\ref{claim:longpath}. 
Then for each level $i \in \mathcal{A}$, let  
$S_{i,1} , \dots , S_{i,m_i}\in \mathcal{V}_i$ be $m_i$ distinct subsets such that there is an edge of $G$ going between $S_{i,j}$ and $S_{i,j+1}$.
For each $i\in \mathcal A$, 
let $\mathcal{V}_i' := \{S_{i,1} , \dots , S_{i,m_i} \}$ and $V_i' := \bigcup_{j=1}^{m_i} S_{i,j}$. 
As $n^2p=\omega(1)$, we may assume that $\ell=\frac{96C}{np}\le \frac{1}{36}n$.
Since each $S \in \mathcal{V}_i$ has size at most $2u_i$ and $m_i \geq n_i / 5$,  \eqref{eqn:c_i} and~\eqref{eqn:size} imply
\begin{align}
|V_i'| &\geq \frac{1}{10}|V_i| \geq \frac{n}{500 \log_2 \left (n^2 p \right )}, \label{eqn:v_i'}\\
\sum_{i \in \mathcal{A}}|V_i'| &\geq \frac{1}{10} \left ( n - \ell - \sum_{i \notin \mathcal{A}}|V_i| \right ) \geq \frac{n}{18}. \label{eqn:sumv_i'} 
\end{align}

\begin{CLAIM}\label{claim:merge}
Whp there exist distinct $T_{1},\dots, T_{s'} \in \mathcal{F}$ satisfying the following.
\begin{enumerate}[label=\rm (\alph*)]
\item $\sum_{i\in [s']} \abs{T_i} \geq n/20$; 
\item For each $j\in [s'-1]$, there exists an edge of $G$ between $T_{j}$ and $T_{j+1}$.
\end{enumerate}
\end{CLAIM}
\begin{subproof}[Proof of Claim 2]
So far, we have fixed the edges of $G[V_i]$ for all $i \in \mathcal{A}$. 
In this claim, we will only use the edges of $G[V_i , V_{i'}]$ for distinct $i,i' \in \mathcal{A}$. This will ensure that each edge of $G$ is exposed at most once.

For each $i\in \mathcal A$, let $a_i$ be the minimum such that $L_i^1:=\bigcup_{j=1}^{a_i} S_{i,j}$ has at least $\frac{1}{20}\abs{V_i'}$ vertices.
Similarly, let $b_i$ be the maximum such that 
$L_i^2:=\bigcup_{j=b_i}^{m_i}S_{i,j}$ has at least $\frac{1}{20}\abs{V_i'}$ vertices.
Then $a_i \leq b_i$.
To see this, if $a_i \geq b_i + 1$, then 
\begin{align*}
    \abs{V_i'} = \sum_{k=1}^{m_i} \abs{S_{i,k}} = \sum_{k=1}^{b_i - 1} \abs{S_{i,k}} + \abs{S_{i,b_i}} + \sum_{k=b_i + 1}^{m_i} \abs{S_{i,k}} < \abs{V_i'}/20 +  \abs{S_{i,b_i}} + \abs{V_i'}/20,
\end{align*}
since $b_i - 1 \leq a_i - 2$. Thus, $\abs{S_{i,b_i}} > 9\abs{V_i'}/10$, which contradicts that $a_i$ is the minimum such that $\bigcup_{j=1}^{a_i} S_{i,j}$ has at least $\abs{V_i'}/20$ vertices, as $\bigcup_{j=1}^{a_i - 1} S_{i,j}$ already has at least $9 \abs{V_i'}/10$ vertices.

Clearly, $\sum_{j=a_i}^{b_i} \abs{S_{i,j}}\ge \abs{V_i'}-
\sum_{j=1}^{a_i-1} \abs{V_i'}
-\sum_{j=b_i+1}^{m_i} \abs{V_i'}
\ge \frac{9}{10}\abs{V_i'}$ for each $i\in\mathcal A$.

Then by~\eqref{eqn:v_i'}, we have $|L_i^1|,|L_i^2| \geq \frac{n}{10000 \log_2 \left (n^2 p \right )}$. 
For $i,i' \in \mathcal A$ with $i < i'$,
exposing the edges of $G[V_{i}, V_{i'}]$,
the probability that $G$ has no edge between $L_i^2$ and $L_{i'}^1$ is at most 
\[
(1 - p)^{|L_i^2||L_{i'}^1|} \leq  \exp\left(- \frac{ n^2 p}{10000^2 \cdot (\log_2 \left (n^2 p \right ))^2} \right) . 
\]

Let $\mathcal{A} = \{i_1 , \dots , i_t \}$ with $i_1 < \dots < i_t$.
As $\abs{\mathcal{A}} \leq \lceil \log_2 \Delta\rceil \leq \log n^2 p$ and $n^2p=\omega(1)$, a union bound implies that whp there is an edge of~$G$  between $L_{i_j}^2$ and $L_{i_{j+1}}^1$ for all $j\in [t-1]$. 
Hence whp, for each $j\in [t]$, there exist $b_{i_j} \leq \beta_j \leq m_{i_j}$ and $1 \leq \alpha_{j+1} \leq a_{i_{j+1}}$  such that 
\begin{equation}\label{stat:prob}
\text{$G$ has an edge between $S_{i_j , \beta_j}$ and $S_{i_{j+1} , \alpha_{j+1}}$}.
\end{equation}
Let $T_1,\dots, T_{s'}$ be the following sequence of subsets in $\mathcal F$ 
\[ 
    S_{i_1,1},\dots, S_{i_1,\beta_1}, S_{i_2, \alpha_2}, S_{i_2,\alpha_2+1},\dots, S_{i_2,\beta_2}, S_{i_3,\alpha_3},\dots, S_{i_{t-1},\beta_{t-1}}, S_{i_t, \alpha_t}, S_{i_t,\alpha_{t}+1},\dots, S_{i_t, m_t}.
\] 
Then we have 
\[
\sum_{j=1}^{s'} \abs{T_{j}} \geq \sum_{i \in \mathcal{A}}  \frac{9}{10} \abs{V_{i}'} \overset{\eqref{eqn:sumv_i'}}{\geq} \frac{n}{20},
\]
and by~\eqref{stat:prob}, $G$ has an edge between $T_{j}$ and $T_{j+1}$ for all $j\in [s'-1]$, as desired.
\end{subproof}

Let $T_{1},\dots, T_{s'} \in \mathcal F$ be  subsets provided by Claim~\ref{claim:merge}. 
For each $1 \leq i \leq s'$, the graph $H[T_i]$ is a connected subgraph of $H$ of at most $\ell \Delta$ vertices, and whp there exists an edge of $G$ between $T_j$ and $T_{j+1}$ for every $1 \leq j < s'$. 

Recall that $R = H \cup G$. By Lemma~\ref{lem:decomp_conn}, $R[T_1 \cup \dots \cup T_{s'}]$ contains vertex-disjoint connected subgraphs $R_1 , \dots , R_m$ where $\ell \Delta \leq \abs{V(R_i)} < 2 \ell \Delta$ for all $i \in [m]$ and 
\begin{align*}
    m \cdot (2 \ell \Delta) \geq \sum_{i=1}^{m} \abs{V(R_i)} > \sum_{i=1}^{s'} \abs{T_i} - \ell \Delta \overset{\text{Claim}~\ref{claim:merge}}{\geq}
    \frac{n}{20} - \ell \Delta.
\end{align*}
Thus, since $\ell=\frac{96C}{np} $, we have
\begin{align*}
    m \geq \frac{1}{2\ell\Delta} \left (\frac{n}{20}-\ell\Delta \right ) = \frac{n^2 p}{3840 C \Delta} - \frac{1}{2} \geq 
    \frac{n^2 p}{9600 C \Delta},
\end{align*}
where the last inequality follows from $\Delta \leq n^2 p (4800 C)^{-1}$. This completes the proof of Lemma~\ref{lem:main}.
\end{proof}

\section{Sharpness of the results}\label{sec:sharpness}
In this section, we will show that the results are best possible. 
To this end, we need the following two lemmas, the first of which is straightforward.
\begin{LEM}\label{lem:contract_forest}
Let $G$ be a graph, and let $V_1 , \dots , V_t \subseteq V(G)$ be disjoint subsets such that
\begin{itemize}
\item $G[V_i]$ is a tree for each $i\in [t]$; 
\item there is at most one edge between $V_i$ and $V_j$ for every pair of distinct $i,j \in [t]$,
\end{itemize}
and let $G^*$ denote  the graph obtained from $G$ by contracting each $V_i$ to a vertex. If $G^*$ is a forest, then so is~$G$.
\end{LEM}

\begin{LEM}\label{lem:example}
Let $C>0$ be an absolute constant, let $p = p(n) \in [0,1]$, and let $1 \leq x = x(n) \leq n$ be an integer-valued function such that $x=o(n)$ and $npx = o(1)$. Let $t$ be an integer such that $t \leq Cn/x$, and let $B_1 , B_2 , \dots , B_t$ be vertex-disjoint trees on at most $x$ vertices. Let $H_0$ be the disjoint union of $B_1 , \dots , B_t$, and let $R_0 := H_0 \cup G(n,p)$. Then whp $R_0$ is a forest. 
\end{LEM}
\begin{proof}
We first claim that whp $R_0[B_i]$ is a tree for each $i \in [t]$. It suffices to prove that whp no edge of $G(n,p)$ lies inside $V(B_i)$ for any $i\in [t]$.
To show this, for each $i\in [t]$ let $\mathcal E_i$ denote the event that there is an edge of $G(n,p)$ that lies inside $V(B_i)$; then $\mathbb P(\mathcal E_i) \leq px^2 / 2$, because the expected number of edges of $G(n,p)$ that lie inside $V(B_i)$ is at most $p \abs{V(B_i)}^2 / 2 \leq px^2 / 2$. Thus, by a union bound we have  
\[ \mathbb{P} \left(\bigcup_{i=1}^t \mathcal E_i \right) \le t \cdot \frac{px^2}{2} \leq \frac{Cnpx}{2} = o(1).\]

Next, we claim that whp $R_0$ has at most one edge between $B_i$ and $B_j$ for all $i \neq j \in [t]$. Equivalently, we show that whp there exists at most one edge of $G(n,p)$  between $B_i$ and $B_j$ for all $i\neq j\in [t]$.
For any distinct pairs $(u_1 , v_1), (u_2 , v_2) \in V(B_i) \times V(B_j)$, let $\mathcal{E}_{u_1, v_1, u_2, v_2}$ be the event that both $u_1 v_1$ and $u_2 v_2$ appear in $G(n,p)$. 
Then, by a union bound we obtain \[ \mathbb{P} \left(\bigcup_{(u_1,v_1) \ne (u_2,v_2)} \mathcal{E}_{u_1,v_1,u_2,v_2}\right) \le  \abs{V(B_i)}^2 \abs{V(B_j)}^2 p^2 \leq x^4 p^2. \]
Therefore the probability that $G(n,p)$ has at least two edges between $B_i$ and $B_j$ is at most $x^4p^2$.
Since there are $\binom{t}{2}$ unordered pairs $i\neq j\in [t]$, the probability that $G(n,p)$ has two edges between $B_i$ and $B_j$ for some $i\neq j\in [t]$ is at most 
\[\binom{t}{2} x^4 p^2 \leq 
\frac12 \left(\frac{Cn}{x}\right)^2 x^4 p^2 = \frac12 (Cnpx)^2 = o(1).\]

Let $R_0^*$ be the graph obtained from $R_0:= H_0 \cup G(n,p)$ by contracting $B_i$ into a single vertex $v_i$ for each $i\in [t]$.
Because the probability that $G(n,p)$ has an edge between $B_i$ and $B_j$ is at most
\begin{equation}\label{eqn:probpart}
1 - (1 - p)^{|B_i||B_j|} \leq 1 - (1 - p |B_i||B_j|) = p |B_i||B_j| \leq px^2, 
\end{equation}
we have $v_i v_j \in E(R_0^*)$ with probability at most $px^2\leq npx = o(1)$. Since $R_0^*$ has at most $t$ vertices, the expected number of cycles of length $i$ in $R_0^*$ is at most
\[ t^i \cdot \left(px^2 \right)^i = \left(tpx^2 \right)^i \leq \left(Cnpx \right)^i. \]
Therefore, the expected number of cycles in $R_0^*$ is at most $\sum_{i \geq 3} (Cnpx)^{i} = o(1)$. Thus, whp $R_0^*$ is a forest, and by Lemma~\ref{lem:contract_forest}, whp $R_0$ is also a forest.
\end{proof}

\subsection{Examples}

The following two examples show that it is necessary to assume $\Delta = O\left (n^2 p \right )$ in Theorem~\ref{thm:fragile-main}, when $p \leq \varepsilon/n$ for some absolute constant $\varepsilon \in (0,1)$.

\begin{EX}\label{example1}
    Assume $p = \Omega(1/n)$ with $p \leq \varepsilon / n$ for some $\varepsilon \in (0,1)$.
    Let $H$ be an $n$-vertex star. Then we have $\Delta(H)=n-1$ and $p=\varepsilon/n$ and whp $G(n,p)$ is outerplanar, and thus $H \cup G(n,p)$ is planar.
\end{EX}

\begin{EX}\label{example2}
    Assume $\Delta = \omega\left( n^2 p \right)$ and $p = o(1/n)$. 
    Let $c(n)$ be a function with $c=c(n)=\omega(1)$, and assume that $n^2 p =\omega(1)$ and $p \leq 1/(cn)$.
    
    Let $cn^2 p /4 \leq t = t(n) \leq cn^2 p /2$ be an integer. Then $t \leq n/2$.
    Let $B_1 , \dots , B_t$ be vertex-disjoint stars on either $\left \lfloor \frac{n-1}{t} \right \rfloor$ or $\left \lceil \frac{n-1}{t} \right \rceil$ vertices such that 
    each $B_i$ has a center
    vertex $r_i$, and 
    \begin{align*}
        \abs{V(B_1)} + \dots + \abs{V(B_t)} = n - 1.   
    \end{align*}
    
    Let $x = x(n) = \left \lceil \frac{4}{cnp} \right \rceil$.
    Then $\abs{V(B_i)} \leq \lceil n/t \rceil \leq x$ for each $i \in [t]$.
    Since $\frac{4}{cnp}\ge 4$, we have $x\le \frac{5}{4}\frac{4}{cnp}$
    and so $t\le \frac{cn^2p}{2} \le \frac{5n}{2x}$.   
    Let $H$ be an $n$-vertex rooted tree obtained from the disjoint union of $B_1,\ldots,B_t$ by adding a root vertex $r$ that is adjacent to $r_1 , \dots , r_t$. 
    Let $R := H \cup G(n,p)$. 
    As $p\leq 1/(cn)$, $n^2p = \omega(1)$, and $x \leq 5/(cnp)$, we have
    \begin{equation*}
    1 \leq x \le \frac{5n}{c}\frac{1}{n^2p}=o(n)\text{\:\: and \:\:}npx \le \frac{5}{c} = o(1).
    \end{equation*}
    By Lemma~\ref{lem:example}, we deduce that whp $R-r$ is a forest. 
    Thus, whp
    $\tw(R) \leq \tw(R-r) + 1 \leq 2$, 
    $g(R) = 0$, and $h(R) \leq h(R-r) + 1 \leq 3$.
\end{EX}

The following two examples show that Theorem~\ref{thm:fragile-main} is best possible for tree-width and the Hadwiger number up to a $\log \left (n^2 p \right )$ factor when $p \leq \varepsilon/n$ for some absolute constant $\varepsilon \in (0,1)$.

\begin{EX}\label{example3}
    Let $\Delta\ge 3$ be an integer. 
    Assume $p = \Omega(1/n)$ with $p \leq \varepsilon / n$ for some $\varepsilon \in (0,1)$. 
    Let $H$ be an $n$-vertex tree obtained from a path on $\lceil  n/\Delta \rceil$ vertices by attaching either $\Delta - 1$ or $\Delta - 2$ leaves to each vertex of the path. Then the maximum degree of $H$ is either $\Delta$ or $\Delta + 1$. Let $R := H \cup G(n,p)$. We claim that whp 
    \[\tw(R) \leq 3 + n/\Delta 
    \quad\text{and}\quad
    h(R) \leq 4 + n/\Delta.\]
    Let $L$ be the set of leaves in $H$. It is well known~\cite{erdos1960} that whp every connected component of $R[L]$ in $G(n,p)$ has at most one cycle, hence $\tw(R[L]) \leq 2$. 
    Since deleting a vertex can decrease its tree-width by at most $1$, 
    whp $\tw(R) \leq \tw(R[L]) + \lceil n/\Delta \rceil \leq 3 + n/\Delta$. 
    The upper bound on $h(R)$ follows from the inequality that $h(R) \leq \tw(R) + 1$.
\end{EX}

\begin{EX}\label{example4}
    Assume $p = o(1/n)$.
    Let $c=c(n)$ be a function with $c(n)=\omega(1)$ and let $3 \leq \Delta = O \left (n^2 p \right )$ be an integer. We also assume that $n^2 p = \omega(1)$ and $p \leq 1/(cn)$.

    Let $cn^2 p /4 \leq t = t(n) \leq cn^2 p /2$ be an integer. Then $t \leq n/2$.   
    Let $P$ be a path on $\left \lceil \frac{t}{\Delta - 2} \right \rceil$ vertices
    and let $v_1 , \dots , v_{\lceil t/(\Delta-2)  \rceil}$ be the vertices of $P$.
    Let $B_1 , \dots , B_{t}$ be trees 
    on either $\left \lfloor \frac{n - \abs{V(P)}}{t} \right \rfloor$ or $\left \lceil \frac{n - \abs{V(P)}}{t} \right \rceil$ vertices, each having maximum degree at most $3$ and  
    satisfying 
    \begin{align*}
        \abs{V(B_1)} + \cdots +\abs{V(B_t)} = n - \abs{V(P)}.
    \end{align*} 

    Let $x = x(n) = \left \lceil \frac{4}{cnp} \right \rceil$. 
    Then $\abs{V(B_i)} \leq \lceil n/t \rceil \leq x$ for each $i \in [t]$.
    Since $\frac{4}{cnp}\ge 4$, we have $x\le \frac{5}{4}\frac{4}{cnp}$
    and so $t\le \frac{cn^2p}{2} \le \frac{5n}{2x}$.
    For each $i\in [t]$, let $u_i \in V(B_i)$ be a leaf of $B_i$. Now we partition $[t]$ into $\abs{V(P)} = \lceil t / (\Delta-2) \rceil$ sets, $I_1 , \dots , I_{\lceil t/(\Delta -2)\rceil}$ such that $|I_i| \leq \Delta-2$ for each $1\leq i \leq \lceil t/(\Delta-2) \rceil$.  Let $H$ be an $n$-vertex tree obtained from the disjoint union of $B_1,B_2,\ldots,B_t$, and $P$ by adding edges~$v_i u_j$ for all $1 \leq i \leq \lceil t/(\Delta-2) \rceil$ and $j \in I_i$.

    Then $B_1,\dots, B_t$ are the connected components of $H-V(P)$ and the maximum degree of $H$ is either $\Delta-1$ or $\Delta$, since $\abs{I_i}= \Delta - 2$ for some $i \in [t]$.
    Moreover, as $p \leq 1/(cn)$, 
    $n^2p =\omega(1)$, and $x \leq 5/(cnp)$, 
    we have
    \begin{equation*}
    1 \leq x \leq \frac{5n}{c} \frac{1}{n^2 p} = o(n)\text{\:\:and\:\:}npx \leq \frac{5}{c} = o(1).
    \end{equation*}
    Let $R := H \cup G(n,p)$. By Lemma~\ref{lem:example}, we deduce that whp $R-V(P)$ is a forest. 
    We may assume that $n$ is sufficiently large so that $cn^2p/\Delta >1$, because $c=\omega(1)$ and $\Delta=O\left (n^2p \right )$.
    Hence whp $R$ satisfies
    \begin{align*}
    \tw(R) & \leq \abs{V(P)} + \tw(R - V(P)) \leq 
    \left\lceil \frac{3t}{\Delta}\right\rceil + 1 
    \leq \left\lceil \frac{3cn^2p}{2\Delta}\right\rceil + 1 
    \leq 
    \frac{3c n^2 p}{\Delta} + 1,\\
    h(R) & \leq \tw(R) + 1 \leq \frac{3c n^2 p}{\Delta} + 2.
    \end{align*}
\end{EX}

\subsection{Sharpness of Theorem~\ref{thm:fragile-main}}

Examples~\ref{example1}--\ref{example4} provide best possible bounds for genus, tree-width,  and Hadwiger number for various ranges of $p$ and $\Delta$.

\begin{enumerate}[label=\rm (\alph*)]
\item Examples~\ref{example1} and~\ref{example2} show that the maximum degree bound $\Delta = O\left (n^2 p \right )$ is necessary for Theorem~\ref{thm:fragile-main}.

\item The lower bound in Theorem~\ref{thm:fragile-main}(a) is best possible by Examples~\ref{example3} and~\ref{example4}; one cannot improve Theorem~\ref{thm:fragile-main} (a) to obtain  
tree-width $\omega\left (n^2 p / \Delta \right )$ when $p \leq \varepsilon/n$ for some $\varepsilon \in (0,1)$.

\item The lower bound in Theorem~\ref{thm:fragile-main}(b) is best possible when $\Delta = O\left (\sqrt{n^2 p} \right )$. Indeed, since $G(n,p)$ has $O \left (n^2 p \right )$ edges whp and the genus increases by at most one when adding an edge, it follows that $g(R) \leq g(H) + O\left (n^2 p \right ) = \Theta \left (\max\left (g(H) , n^2 p \right ) \right )$ whp.

\item The bound in Theorem~\ref{thm:fragile-main}(c) is best possible up to logarithmic factor in $n$.
Consider a connected graph $H$ of genus $O\left (n^2 p \right )$.
Since $g(R) \leq \Theta \left (\max\left (g(H) , n^2 p \right ) \right )$ whp and the genus of $K_k$ is $\left \lceil \frac{(k-4)(k-3)}{12} \right \rceil = \Theta\left(k^2 \right)$, whp we have
\[h(R) \leq \Theta (\sqrt{g(R)}) = \Theta\left(\max \left (\sqrt{g(H)} , \sqrt{n^2 p} \right )\right) = O\left (\sqrt{n^2 p} \right ).\]
When $\Delta = O(1)$, this shows that Theorem~\ref{thm:fragile-main}(c) is best possible.
When $\Delta = \Omega\left (\sqrt{n^2 p \log \left (n^2 p \right )} \right )$, we can consider the graph $H$ in Example~\ref{example4}, which ensures whp $h(R) = \Theta\left (n^2p/\Delta \right )$ and shows that Theorem~\ref{thm:fragile-main}(c) is best possible.
\end{enumerate}

\subsection{Sharpness of Theorem~\ref{thm:spanningforest}}\label{subsec:sharp_spanningforest}

The following observation shows that Theorem~\ref{thm:spanningforest} would fail if the path cover number of $H$ is $\omega\left(n^2 p \right)$. 

\begin{OBS}\label{obs:spanningforest_fail}
Let $c = c(n)$ be a function with $c(n) = \omega(1)$, and let $p = p(n)$ be a function with $p \leq 1/(cn)$ and $n^2 p = \omega(1)$.
Then there exists an $n$-vertex graph $H$ with the path cover number at most $O\left (cn^2 p \right )$ and maximum degree $o\left (n^2 p \right )$ such that whp both $\tw(H \cup G(n,p))$ and $h(H \cup G(n,p))$ are at most~$o(c)$.
\end{OBS}
\begin{proof}
Let $c^* = c^*(n)$ be a slowly increasing function such that $c^* = \omega(1)$, $c^* = o\left (n^2 p \right )$, and $\left(c^* \right)^2 = o(c)$. 
Let $\Delta^* = \lceil n^2 p / c^* \rceil$, and let $H$ be an $n$-vertex tree in Example~\ref{example4} such that $B_1 , \dots , B_t$ are paths, and $c^*$ and $\Delta^*$ play the roles of $c$ and $\Delta$ respectively. Note that $p \leq 1/(cn) \leq 1/\left (c^* n \right)$ for all sufficiently large $n$.

Then $H$ has maximum degree at most $\Delta^* = o\left (n^2 p \right )$ and has exactly $t = \Theta\left (c^* n^2 p \right )$ leaves. Thus, the path cover number of $H$ is $\Theta \left (c^* n^2 p \right )$ (see~\cite[Claim 1.2]{GOVW2019}), which is at most $O\left (c n^2 p \right )$.
However, whp both $\tw(H \cup G(n,p))$ and $h(H \cup G(n,p))$ are at most $\frac{3 c^* n^2 p}{\Delta} + 2 = o(c)$, as $\left(c^* \right)^2 = o(c)$.
\end{proof}

\subsection{Sharpness of Lemma~\ref{lem:main}}
The key idea of Lemma~\ref{lem:main} is to obtain $m$ vertex-disjoint connected subgraphs $R_1 , \dots , R_m$ so that even though we begin with a random graph $G(n,p)$ in a subcritical regime $np < \varepsilon$ for some absolute constant $\varepsilon \in (0,1)$, whp $H \cup G(n,p)$ contains $G(m,q)$ in a supercritical regime $mq > 1 +  \varepsilon'$ as a minor, for some absolute constant $\varepsilon' > 0$ and some $q = 1 - (1 - p)^{-\Theta\left(m^2 \right)}$. 

Since $mq = \Omega(1)$ implies $m = O\left (n^2 p \right )$, one might think that the result $m = \Theta\left (n^2 p / \Delta \right)$ in Lemma~\ref{lem:main} is weak if $\Delta = \omega(1)$. 
However, we claim that the condition $m = \Theta\left (n^2 p / \Delta \right )$ is essentially best possible, as follows.

\begin{OBS}
Let $p = p(n) \leq \varepsilon/n$ for some absolute constant $\varepsilon \in (0,1)$ with $n^2 p = \omega(1)$.
Let $\zeta = \omega(1)$ be sufficiently smaller than $n^2 p$, and let $\Delta = O\left (n^2 p \right )$ which is  sufficiently larger than $\zeta$. 

If $p = \Omega(1/n)$, then let $H$ be the $n$-vertex tree in Example~\ref{example3} and let $p' = \frac{1 + \varepsilon}{2} - p$. 
Otherwise, if $p = o(1/n)$ then let~$H$ be the $n$-vertex tree in Example~\ref{example4}, where $c = c(n)$ is chosen to satisfy $c = o(\zeta)$, and let $p' = p$.

For $m = \Theta\left (\zeta n^2 p / \Delta \right )$, 
whp $H' = H \cup G(n,p)$ does not have $m$ vertex-disjoint connected subgraphs $R_1 , \dots , R_m$ with $|V(R_i)| = \Theta(n/m)$ for all $i \in [m]$.
\end{OBS}
\begin{proof}
Observe that $p' = \Theta(p)$ and $p+p' \leq \frac{1+\varepsilon}{2n}$ for $\frac{1 + \varepsilon}{2} \in (0,1)$.
For the contrary, suppose that whp $H' = H \cup G(n,p)$ has $m$ vertex-disjoint connected subgraphs $R_1 , \dots , R_m$ with $|V(R_i)| = \Theta(n/m)$ for all $i \in [m]$.
Let $p^* = 1 - (1 - p)(1 - p') \leq p + p'$. We regard $H \cup G(n,p^*) = H \cup G(n,p) \cup G(n,p') = H' \cup G(n,p')$, where $G(n,p)$ and $G(n,p')$ are independent. Then for every $i \ne j \in [m]$, the probability that there exists an edge of $G(n,p')$ between $V(R_i)$ and $V(R_j)$  is
$ 1 - (1 - p')^{|V(R_i)||V(R_j)|}$.

Let \[ q=\min_{i\neq j\in[m]} \left( 1 - (1 - p')^{|V(R_i)||V(R_j)|}\right).\]
If $p' |V(R_i)| |V(R_j)|  < 1$, then 
\[
    1 - (1 - p')^{|V(R_i)||V(R_j)|} \geq 1 - e^{-p'|V(R_i)||V(R_j)|} \geq (1 - e^{-1}) p'|V(R_i)||V(R_j)| = \Theta \left (\frac{\Delta^2}{\zeta^2 n^2 p} \right ).
\] 
If $p' |V(R_i)| |V(R_j)|  \geq 1$, then 
\[ 1 - (1 - p')^{|V(R_i)||V(R_j)|} \geq 1 - e^{-1}.\]
Therefore, $q=\Theta\left (\min\left (1-e^{-1}, \frac{\Delta^2}{\zeta^2 n^2 p} \right ) \right )$.
Note that $m=\omega(1)$
and 
$m \frac{\Delta^2}{\zeta^2 n^2 p} = \Theta(\Delta/\zeta)>2$ as $\Delta$ is chosen sufficiently larger than $\zeta$.
Thus we have $mq>2$.

Thus, we may consider $H' \cup G(n,p')$ containing  $G(m,q)$ as a minor.
By Theorem~\ref{thm:randomtw}, whp $\tw(G(m,q)) = \Omega(m)$, and therefore  
\begin{equation}\label{eqn:obs_lowerbdd}
    \tw(H \cup G(n,p^*)) = \Omega(m)= \Omega\left(\frac{\zeta n^2 p}{\Delta}\right).
\end{equation}

Note that $p^* \leq p + p' \leq \frac{1 + \varepsilon}{2n}$. If $p^* = \Omega(1/n)$, then Example~\ref{example3} shows that whp $\tw(H \cup G(n,p^*)) = O\left (n^2 p / \Delta \right )$, contradicting~\eqref{eqn:obs_lowerbdd}. On the other hand, if $p^* = o(1/n)$ then Example~\ref{example4} shows that whp $\tw\left (H \cup G\left(n,p^* \right) \right) = O\left (cn^2 p / \Delta \right )$, contradicting~\eqref{eqn:obs_lowerbdd} as we chose $c = o(\zeta)$.
\end{proof}

\section{Discussions}\label{sec:discussions}

In Theorem~\ref{thm:fragile-main}(b), if the maximum degree $\Delta$ of the base graph $H$ is $O\left (\sqrt{n^2 p} \right )$, then the lower bound $\Omega\left (n^2 p \right )$ for the genus of  $R$ is best possible. However, we could not prove whether our bound is best possible or can be improved when $\Delta = \Omega\left (\sqrt{n^2 p} \right )$.  

\begin{PROB}
Determine the asymptotic behaviour of $g(R)$ in Theorem~\ref{thm:fragile-main}(c) when $\Delta=\omega\left(\sqrt{n^2 p}\right)$.
\end{PROB}

To obtain the lower bound in Theorem~\ref{thm:fragile-main}(c), we found many connected subgraphs of comparable sizes and estimated the genus of a minor obtained by contracting these connected subgraphs. This strategy is effective for $\Delta = O\left(\sqrt{n^2 p}\right)$, as each of these connected subgraphs contains only a few edges, which does not significantly affect our estimation. However, when $\Delta = \Omega\left(\sqrt{n^2 p}\right)$, there may be many edges in each of these connected subgraphs, suggesting that a novel method is needed to account for these edges.

A curious reader might ask whether the results corresponding to Theorems~\ref{thm:fragile-main}(b) and~\ref{thm:spanningforest}(b) also hold for~\emph{non-orientable genus}, where the non-orientable genus of a graph $G$ is the minimum number of cross-caps to be added to the sphere such that $G$ can be embedded without any crossings.
Indeed, Theorem~\ref{thm:randomgenus} holds for non-orientable genus as well (see~\cite[Section 8]{dowden2019}), and therefore, we can obtain bounds similar to Theorems~\ref{thm:fragile-main}(b) and~\ref{thm:spanningforest}(b) for non-orientable genus.

\paragraph{Acknowledgement.} 
The authors would like to thank the anonymous reviewers for their careful reading and helpful comments. 
Part of this work was conducted when Mihyun Kang visited the Institute for Basic Science (IBS) and her special thanks go to the Discrete Mathematics Group of the IBS.

\providecommand{\bysame}{\leavevmode\hbox to3em{\hrulefill}\thinspace}
\providecommand{\MR}{\relax\ifhmode\unskip\space\fi MR }
\providecommand{\MRhref}[2]{%
  \href{http://www.ams.org/mathscinet-getitem?mr=#1}{#2}
}
\providecommand{\href}[2]{#2}

\newpage
\appendix
\section*{Appendix. Minor-monotone graph parameters of random graphs}
The following table summarizes interesting results of minor-monotone graph parameters of random graphs.
\begin{table}[h]
\begin{center}
\footnotesize
\renewcommand\arraystretch{1.3}
\begin{tabular}{ cp{0.35\textwidth}p{0.35\textwidth}c  }
  \toprule
\textbf{Parameters} & \textbf{Values in $G(n,p)$ (whp)} & \textbf{Range of $p$} & \textbf{Ref}\\
\midrule
\multirow{2}{0.1\textwidth}{Tree-width $\tw$} 
& $\tw \leq 2$ & $p \leq c/n$ $(0<c<1)$ & Folklore\\ \cline{2-4}
& $\tw = \Theta(n)$ & $p \geq c/n$ $(c>1)$ & \cite{lee2012} \\ 
\midrule
\multirow{7}{0.1\textwidth}{Genus $g$} 
& $g=0$ & $p = n^{-1} - \omega(n^{-4/3})$ & \cite{luczak1994} \\ \cline{2-4}
& $g = (1+o(1)){c^3 n^4}/{3}$ & $p = n^{-1} + c$ and $n^{-4/3} \ll s \ll n^{-1}$ & \cite{dowden2019}\\ \cline{2-4}
& $g = (1+o(1)) \mu(c) {n^2 p}/{2}$ & $np=c$ and $c>1$, where $\lim_{c \to 1} \mu(c)=0$ and $\lim_{c \to \infty}\mu(c) = 1/2$ & \cite{dowden2019} \\\cline{2-4} 
& $(1-o(1)) {n^2p}/{4} \leq g \leq {n^2 p}/{4}$ & $1 \ll np = n^{o(1)}$ & \cite{dowden2019} \\ \cline{2-4}
& $(1+o(1)) \max \left( \frac{1}{12}, \frac{(j-1)}{4(j+1)} \right) n^2 p $\hfill \mbox{}\linebreak $\leq g \leq (1+o(1)) \frac{jn^2 p}{4(j+2)}$ & $np = \Theta(n^{1/j})$ & \cite{rodl1995} \\ \cline{2-4}
& $g = (1+o(1)) \frac{jn^2 p}{4(j+2)}$ & $n^{\frac{1}{j+1}} \ll np \ll n^{\frac{1}{j}} \ (j\in \mathbb N)$ & \cite{rodl1995} \\ \cline{2-4}
& $g = (1+o(1)) {n^2 p}/{12}$ & $p = \Theta(1)$ & \cite{rodl1995} \\ 

\midrule
\multirow{3}{0.1\textwidth}{Hadwiger number $h$} 
& $h = \Theta(c(n)^{3/2})$ & $p = n^{-1} + c(n) n^{-4/3}$, where $c(n) = \omega(1)$ but $c(n) = o(n^{1/3})$ & \cite{fountoulakis2009}\\ \cline{2-4}
& $\delta(c) \sqrt{n} \leq h \leq 2 \sqrt{cn}$ for some constant $\delta(c)$& $p \geq c/n$ and $c>1$
& \cite{fountoulakis2008} \\ \cline{2-4}
& $\frac{(1 - \varepsilon)n}{\sqrt{\log_{1/(1-p)} (np)}} \leq h \leq  \frac{(1 + \varepsilon)n}{\sqrt{\log_{1/(1-p)} (np)}}$ & $C(\varepsilon)/n \leq p \leq 1 - \varepsilon$ & \cite{bollobas1980, fountoulakis2008} \\
\bottomrule
\end{tabular}
\end{center}
\caption{Summary of minor-monotone parameters of random graphs.} \label{fig:table1}
\end{table}
\end{document}